\documentclass{amsart}

\usepackage{amssymb, verbatim,amsmath}
\usepackage{enumerate}

\usepackage{color}
\usepackage{soul}

\newtheorem{theorem}{Theorem}
\newtheorem{proposition}{Proposition}
\date{\today}

\newcommand{\hyper}[5]{\,{}_{#1}F_{#2}\left(
\begin{array}{r|l}
\begin{array}{cc}{\displaystyle{#3}}\\
{\displaystyle{#4}}
\end{array} & {\displaystyle{#5}}
\end{array} \right)}

\newtheorem{remark}{Remark}

\newcommand{\phin}[6]{\,{}S_{#1}\left(\!\!%
\begin{array}{c|c}
\begin{array}{cc}{\displaystyle{#2}} & {\displaystyle{#3}}\\[-0.1ex]
{\displaystyle{#4}} & {\displaystyle{#5}} \end{array} &
\,{\displaystyle{#6}}
\end{array} \right)}

\newcommand{\ww}[6]{\,{}W_{#1}\left(\!\!%
\begin{array}{c|c}
\begin{array}{cc}{\displaystyle{#2}} & {\displaystyle{#3}}\\[-0.1ex]
{\displaystyle{#4}} & {\displaystyle{#5}} \end{array} &
\,{\displaystyle{#6}}
\end{array} \right)}

\newcommand{\sss}[6]{\,{}S_{#1}\left(\!\!%
\begin{array}{c|c}
\begin{array}{cc}{\displaystyle{#2}} & {\displaystyle{#3}}\\[-0.1ex]
{\displaystyle{#4}} & {\displaystyle{#5}} \end{array} &
\,{\displaystyle{#6}}
\end{array} \right)}

\newcommand{\ssstar}[6]{\,{}S_{#1}^{*}\left(\!\!%
\begin{array}{c|c}
\begin{array}{cc}{\displaystyle{#2}} & {\displaystyle{#3}}\\[-0.1ex]
{\displaystyle{#4}} & {\displaystyle{#5}} \end{array} &
\,{\displaystyle{#6}}
\end{array} \right)}

\begin{document}

\title[A basic class of discrete orthogonal polynomials]{A basic class of symmetric orthogonal polynomials of a discrete variable}

\author[Masjed-Jamei]{\sc Mohammad Masjed-Jamei}
\address[Masjed-Jamei]{Department of Mathematics, K.N.Toosi University of Technology, P.O. Box 16315--1618, Tehran, Iran.}
\email[Masjed-Jamei]{mmjamei@kntu.ac.ir, mmjamei@yahoo.com}

\author[Area]{\sc Iv\'an Area}
\address[Area]{Departamento de Matem\'atica Aplicada II,
              E.E. de Telecomunicaci\'on,
              Universidade de Vigo,
              Campus Lagoas-Marcosende,
              36310 Vigo, Spain.}
\email[Area]{area@uvigo.es}

\thanks{\emph{Acknowledgments.} The work of M. Masjed-Jamei has been supported by a grant from``Iran National Science Foundation" No. 91002576 and the work of I. Area has been partially supported by the Ministerio de Ciencia e Innovaci\'on of Spain under grants MTM2009--14668--C02--01 and MTM2012--38794--C02--01, co-financed by the European Community fund FEDER. The referee and handling editor deserve special thanks for careful reading and many useful comments and suggestions which have improved the manuscript. Dedicated to Prof. A. Ronveaux on the occasion of his 80th Birthday}

\begin{abstract}
By using a generalization of Sturm-Liouville problems in discrete spaces, a basic
class of symmetric orthogonal polynomials of a discrete variable with four free parameters,
which generalizes all classical discrete symmetric orthogonal polynomials, is introduced. The
standard properties of these polynomials, such as a second order difference equation, an explicit form for the polynomials, a three term recurrence relation and an orthogonality relation are presented. It is shown that two hypergeometric orthogonal sequences with 20 different weight functions can be extracted from this class.  Moreover, moments corresponding to these weight functions can be explicitly computed. Finally, a particular example containing all classical discrete symmetric orthogonal polynomials is studied in detail.
\end{abstract}



\maketitle
\setcounter{section}{0}
\setcounter{equation}{0}

\section{Introduction}
Some special functions of mathematical physics such as classical orthogonal polynomials and cylindrical functions \cite{MR922041}, are solutions of a differential equation of hypergeometric type \cite{MR1149380,MR922041,MR998364}
\begin{equation}\label{eq:11suslov}
\sigma(x)y''(x)+\tau(x)y'(x)+\lambda y(x)=0,
\end{equation}
and extendible by changing equation (\ref{eq:11suslov}) to a difference equation of the form
\begin{equation}\label{eq:12suslov}
\tilde{\sigma}(x(s)) \frac{\Delta}{\nabla x_{1}(s)} \left[ \frac{\nabla y(s)}{\nabla x(s)} \right] + \frac{\tilde{\tau}(x(s))}{2} \left[ \frac{\Delta y(s)}{\Delta x(s)} + \frac{\nabla y(s)}{\nabla x(s)} \right] + \lambda y(s)=0,
\end{equation}
where
\[
\Delta x(s)=x(s+1)-x(s), \quad \nabla x(s)=\Delta x(s-1), \quad \frac{\Delta}{\Delta x(s)}f(s)=\frac{f(s+1)-f(s)}{x(s+1)-x(s)},
\]
$\tilde{\sigma}(x(s))$ and $\tilde{\tau}(x(s))$ are polynomials of degree at most two and one, respectively, in $x(s)$, $\lambda$ is a constant, and $x_{1}(s)=x(s+1/2)$.

The difference equation (\ref{eq:12suslov}), which is obtained by approximating the differential equation (\ref{eq:11suslov}) on a non-uniform lattice, is of much importance \cite{MR998364} as its particular solutions have been applied in quantum mechanics, theory of group representations and especially computational mathematics, where one can point to the Clebsch-Gordan and  Racah coefficients with wide applications in atomic and nuclear spectroscopy. There exist different approaches for the analysis of orthogonal polynomials of a discrete variable running from the classical references \cite{MR0481884,1941} to the recent monograph \cite{MR2656096}, which is a basic reference on orthogonal polynomials.

Also there exists a number of numerical and symbolic methods for solving hypergeometric equations of type (\ref{eq:11suslov}) or (\ref{eq:12suslov}), which are of interest in applications, particularly for cases containing symmetric solutions, such as resolution of the Gibbs phenomenon \cite{MR1491051,MR2726815}, Fourier-Kravchuk transform used in Optics \cite{MR1456607}, approximation of harmonic oscillator wave functions \cite{MR1698479}, tissue segmentation of human brain MRI through preprocessing \cite{Archibald2003489}, reconstructions for electromagnetic waves in the presence of a metal nanoparticle \cite{Min2006730}, efficient determination of the critical parameters and the statistical quantities for Klein-Gordon and sine-Gordon equations with a singular potential \cite{Chakraborty2012}, image representation \cite{Zhu20121540,Hosny2012476} and quantitative theory for the lateral momentum distribution after strong-field ionization \cite{Dreissigacker2012}.

The main aim of this paper is to introduce a basic class of symmetric orthogonal polynomials of a discrete variable with four free parameters, which is the polynomial solution of a symmetric generalization of equation (\ref{eq:12suslov}) on the uniform lattice $x(s)=s$. Computational aspects of these new polynomials are described in detail giving their explicit representation  as well as the three-term recurrence relation they satisfy. A full classification of weight functions and orthogonality supports is given together with computing the moments of the aforesaid weights. From this class all classical symmetric orthogonal polynomials of a discrete variable can be recovered (section \ref{section5}), and its limit relation with the continuous type of generalized classical symmetric orthogonal polynomials is given (see remark \ref{remark1}).

A regular Sturm-Liouville problem of continuous type is a boundary value problem in the form
\begin{equation}\label{eq:1}
\frac{d}{dx} \left( k(x) \frac{dy_{n}(x)}{dx} \right) + \left(\lambda_{n} \varrho(x)-q(x) \right) y_{n}(x)=0 \qquad (k(x)>0, \varrho(x)>0),
\end{equation}
which is defined on an open interval $(a,b)$, and has the boundary conditions
\begin{equation}\label{eq:2}
\alpha_{1} y(a) + \beta_{1} y'(a)=0, \quad \alpha_{2} y(b) + \beta_{2} y'(b)=0,
\end{equation}
where $\alpha_{1}, \alpha_{2}$ and $\beta_{1}, \beta_{2}$, are given constants and $k(x)$, $k'(x)$, $q(x)$, and $\varrho(x)$ in (\ref{eq:1}) are to be assumed continuous for $x \in [a,b]$. If one of the boundary points $a$ and $b$ is singular (i.e. $k(a) = 0$ or $k(b) = 0$), the problem is transformed to a singular Sturm-Liouville problem.

Let $y_{n}$ and $y_{m}$ two eigenfunctions of the operator
$D(k(x)D) -q(x)I$, where $D$ is the standard derivative operator.
According to Sturm-Liouville theory \cite{MR922041}, they are orthogonal with respect to the weight function $\varrho(x)$ under the given conditions (\ref{eq:2}) and satisfy the orthogonality relation
\begin{equation*}
\int_{a}^{b} \varrho(x) y_{n}(x) y_{m}(x) dx =
\left( \int_{a}^{b} \varrho(x) y_{n}^{2}(x) dx \right) \delta_{n,m}.
\end{equation*}
Many of special functions are orthogonal solutions of a regular or singular Sturm-Liouville problem having the symmetry property ($\phi_{n}(-x)=(-1)^{n} \phi_{n}(x)$) so that have found valuable applications in physics and engineering, as already mentioned. In \cite{MR2374588}, the classical equation (\ref{eq:1}) is symmetrically extended as follows.

\begin{theorem}\cite{MR2374588}\label{firstth}
Let $\phi_{n}(-x)=(-1)^{n} \phi_{n}(x)$  be a sequence of symmetric functions satisfying the equation
\begin{equation}\label{eq:nova7}
A(x) \phi_{n}''(x) + B(x) \phi_{n}'(x) + \left( \lambda_{n} C(x) + D(x) + \sigma_{n} E(x)\right) \phi_{n}(x)=0,
\end{equation}
where
\begin{equation}\label{eq:sigman}
\sigma_{n}=\frac{1-(-1)^{n}}{2}=\begin{cases} 0, & n \text{ even,} \\ 1,& n \text{ odd,}\end{cases}
\end{equation}
and $\lambda_{n}$ is a sequence of constants. If $A(x)$, $(C(x)>0)$, $D(x)$ and $E(x)$ are even functions and $B(x)$ is odd then
\begin{equation*}
\int_{-\nu}^{\nu} \varrho^{*}(x) \phi_{n}(x) \phi_{m}(x) dx = \left( \int_{-\nu}^{\nu}\varrho^{*}(x) \phi_{n}^{2}(x) dx \right) \delta_{n,m},
\end{equation*}
where
\begin{equation}\label{eq:6}
\varrho^{*}(x)= C(x) {\exp} \left( \int \frac{B(x)-A'(x)}{A(x)} dx \right) = \frac{C(x)}{A(x)} {\exp} \left( \int \frac{B(x)}{A(x)} dx \right).
\end{equation}
The weight function defined in (\ref{eq:6}) must be positive and even on $[-\nu,\nu]$ and the function
\begin{equation*}
A(x)K(x)=A(x) {\exp} \left( \int \frac{B(x)-A'(x)}{A(x)} dx\right) = {\exp} \left( \int \frac{B(x)}{A(x)} dx\right)
\end{equation*}
must vanish at $x = \nu$ , i.e. $A(\nu) K(\nu) = 0$. In this way, since $K(x) = \varrho^{*}(x) /C(x)$ is an even function so $A(-\nu) K(-\nu) = 0$ automatically.
\end{theorem}

Using this theorem, many symmetric special functions of continuous type have been generalized in \cite{MR2270049, MR2374588,MR2421844,MR2601301,MR2467682,MR2743534,1220.33011}.

Orthogonal functions of a discrete variable can similarly be solutions of a regular or singular Sturm-Liouville problem of discrete type in the form \cite{0817.39004}
\begin{equation}\label{eq:8}
\Delta \left( K^{*}(x) \nabla y_{n}(x) \right) + \left( \lambda_{n} w(x)-q^{*}(x) \right)y_{n}(x)=0 \qquad (K^{*}(x)>0, w(x)>0),
\end{equation}
where
\begin{equation*}
\Delta f(x)=\nabla f(x+1)=f(x+1)-f(x),
\end{equation*}
and (\ref{eq:8}) satisfies a set of discrete boundary conditions like (\ref{eq:2}). This means that if $y_{n}(x)$ and $y_{m}(x)$ are two eigenfunctions of difference equation (\ref{eq:8}), they are orthogonal with respect to the weight function $\varrho^{*}(x)$ on a discrete set \cite{MR1149380}.

Recently in \cite{MASJEDAREA} we have presented the following theorem by which one can generalize usual Sturm-Liouville problems with symmetric solutions in discrete spaces. As a very important consequence of this theorem, we can introduce a basic class of symmetric orthogonal polynomials of a discrete variable with four free parameters.

\vspace*{0.2cm}
\begin{theorem}\cite{MASJEDAREA}\label{th:12}
Let $\phi_{n}(-x)=(-1)^{n} \phi_{n}(x)$  be a sequence of symmetric functions that satisfying the difference equation
\begin{multline}\label{eq:9}
A(x) \Delta \nabla \phi_{n}(x) + \left(A(-x)-A(x) \right) \Delta \phi_{n}(x) \\ + \left( \lambda_{n} C(x) + D(x) + \sigma_{n} E(x) \right) \phi_{n}(x) =0,
\end{multline}
where, as usual, $\Delta \nabla=\Delta-\nabla$. If $A(x)$ is a free real function and $(C(x)>0)$, $D(x)$ and $E(x)$ are even functions then
\begin{equation*}
\sum_{x=\alpha}^{\beta-1} W^{*}(x) \phi_{n}(x) \phi_{m}(x)  = \left[ \sum_{x=\alpha}^{\beta-1} W^{*}(x) \phi_{n}^{2} (x) \right] \delta_{n,m}
\end{equation*}
where
\begin{equation}\label{eq:17}
W^{*}(x)=C(x)W(x)
\end{equation}
and $W(x)$ is the solution of the Pearson difference equation
\begin{equation*}
\Delta \left[A(x)W(x) \right]  = \left[ A(-x)-A(x)\right] W(x),
\end{equation*}
which is equivalent to
\begin{equation}\label{eq:19}
\frac{W(x+1)}{W(x)} = \frac{A(-x)}{A(x+1)}.
\end{equation}
Moreover, the weight function defined in (\ref{eq:17}) must be even over one of the four following symmetric counter sets
\begin{itemize}
\item[i)] $S_{1}=\left \{-a-n,-a-n+1,\dots,-a-1,-a,a,a+1,\dots,a+n \right \}$, $a \in {\mathbf{R}}$,
\item[ii)] $S_{2}=S_{1} \cup \{0\}$ (as any odd function is equal to zero at $x=0$),
\item[iii)] $S_{3}=\left \{\dots, -a-n,-a-n+1,\dots,-a-1,-a,a,a+1,\dots,a+n, \dots \right \}$ (an infinite set)
\item[iv)] $S_{4}=S_{3} \cup \{0\}$,
\end{itemize}
and the function $A(x)W(x)$ must also vanish at $x=\alpha$ and $x=1-\beta$, where $[\alpha,\beta-1] \in \left \{S_{1},S_{2},S_{3},S_{4} \right \}$.
\end{theorem}

\section{A basic class of symmetric orthogonal polynomials of a discrete variable using Theorem \ref{th:12}}

As a special case of equation (\ref{eq:nova7}), the following differential equation is defined in \cite{MR2270049}:
\begin{multline}\label{eq:240}
x^{2}(p^{*}x^{2}+q^{*}) \Phi_{n}''(x) + x(r^{*}x^{2}+s^{*}) \Phi_{n}'(x)-(n(r^{*}+(n-1)p^{*})x^{2}+\sigma_{n} s^{*}) \Phi_{n}(x)=0.
\end{multline}
Here one of the basic solutions is the symmetric class of orthogonal polynomials denoted by
\begin{multline*}
\ssstar{n}{r^{*}}{s^{*}}{p^{*}}{q^{*}}{t}=\sum_{k=0}^{[n/2]} \binom{[n/2]}{k} \\ \times \left( \prod_{i=0}^{[n/2]-(k+1)} \frac{(2i+(-1)^{n+1}+2[n/2])p^{*}+r^{*}}{(2i+(-1)^{n+1}+2)q^{*}+s^{*}} \right) x^{n-2k}.
\end{multline*}
By referring to theorem \ref{firstth}, we observe in (\ref{eq:240}) that $A(x)=x^{2}(p^{*}x^{2}+q^{*})$ is a polynomial of degree at most four, $C(x)=x^{2}$ is a symmetric quadratic polynomial, $D(x) = 0$, and $E(x) = s^*$. Since discrete orthogonal polynomials have a direct relationship with continuous polynomials, motivated by (\ref{eq:240}), we suppose in the main equation (\ref{eq:9}) that
\begin{equation}\label{eq:14}
A(x)=\sum_{i=0}^{4} a_{i}x^{i}, \quad C(x)=c_{2}x^{2}+c_{0}, \quad D(x)=0, \quad E(x)=e_{0},
\end{equation}
though there may exist some other possible cases.

By noting the assumptions (\ref{eq:14}), we are now interested in obtaining a symmetric orthogonal polynomial solution. Hence, let
\begin{equation}\label{eq:29}
\phi_{n}(x)=x^{n}+\delta_{n} x^{n-2} + \cdots,
\end{equation}
satisfy a three term recurrence relation as
\begin{equation}\label{eq:30}
\phi_{n+1}(x)=x \phi_{n}(x)-\gamma_{n} \phi_{n-1}(x), \quad ({\rm{with }}\,\,\, \phi_{0}(x)=1, \,\,\, \phi_{1}(x)=x).
\end{equation}

From (\ref{eq:9}), (\ref{eq:14}) and (\ref{eq:29}), equating the coefficient in $x^{n+2}$ gives
\begin{equation}\label{eq:lambdan}
\lambda_{n}=\frac{n \left(2 a_3-a_4 (n-1)\right)}{c_2},
\end{equation}
provided that $c_{2} \neq 0$ and also $\vert a_{3} \vert + \vert a_{4} \vert \neq 0$.

By using the eigenvalue $\lambda_{n}$ in (\ref{eq:lambdan}) and equating the coefficient in $x^{n}$ we obtain
\begin{multline*}
\delta_{n}=\left\{6 c_2 \left(4 a_1 n-2 a_2 (n-1) n+e_0
   \left((-1)^n-1\right)\right)  \right. \\ \left.  +a_4 (n-1) n  \left(12 c_0-c_2 (n-3)
   (n-2)\right)  \right. \\ \left. +4 a_3 n \left(c_2 (n-2) (n-1)-6 c_0\right) \right\}/\left\{ 24 c_2 \left(a_4 (3-2 n)+2 a_3\right) \right \}.
\end{multline*}

Also from (\ref{eq:29}) and (\ref{eq:30}) we have
\begin{equation*}
x^{n+1} +  \delta_{n+1} x^{n-1} + \cdots = x \left( x^{n} +  \delta_{n} x^{n-2} + \cdots  \right) - \gamma_{n} \left( x^{n-1} +  \delta_{n-1} x^{n-3}  + \cdots \right),
\end{equation*}
which implies
\begin{equation}\label{eq:27new}
\gamma_{n}=\delta_{n}-\delta_{n+1}.
\end{equation}

Therefore, in order that $\phi_{n}(x)$ is a solution of (\ref{eq:9}), the following extra conditions must be considered for the initial values of $n$:
\begin{equation*}
e_{0}=2 a_1-\frac{2 a_3 c_0}{c_2}, \quad a_{0}=\frac{c_0 \left(\left(2 a_3-a_4\right) c_0+\left(a_2-2 a_1\right)   c_2\right)}{c_2^2}, \quad c_{2}=-4 c_0.
\end{equation*}

As a conclusion, we get
\begin{equation}\label{eq:28new}
\begin{cases}
A(x)=\displaystyle{a_4 x^4+a_3 x^3+a_2 \left(x^2-\frac{1}{4}\right)+a_1
   \left(x+\frac{1}{2}\right)+\frac{a_3}{8}-\frac{a_4}{16}} , \\
B(x)=A(-x)-A(x)=-2 x \left(a_3 x^2+a_1\right), \\
C(x)=c_0 \left(1-4 x^2\right), \\
D(x)=0, \\
E(x)=\displaystyle{\frac{1}{2} \left(4 a_1+a_3\right)}, \\[3mm]
\lambda_{n}=\displaystyle{\frac{n \left(a_4 (n-1)-2 a_3\right)}{4 c_0}},
\end{cases}
\end{equation}
and
\begin{multline}\label{eq:29new}
\delta_{n}=\left\{ 12 a_1 \left(2 n+(-1)^n-1\right)+3 a_3 \left((-1)^n-1\right) \right. \\
\left.  +n   \left(-12 a_2 (n-1)+2 a_3 (2 (n-3) n+7) \right. \right. \\ \left. \left. -a_4 (n-1) ((n-5)   n+9)\right) \right\}/\left\{ 24 a_4 (3-2 n)+48 a_3 \right\}.
\end{multline}

For simplicity, if we set
\begin{equation*}
a_{4}=2a, \quad a_{3}=a+2b, \quad a_{2}=b+2c, \quad a_{1}=c+2d,
\end{equation*}
then (\ref{eq:28new}) changes to
\begin{equation*}
\begin{cases}
A(x)=(2x+1)(a x^{3}+b x^{2}+c x+d), \\
B(x)=-2 x \left((a+2 b)x^2 +c+2 d\right), \\
C(x)=\displaystyle{\frac{1}{4}- x^2}, \\
D(x)=0, \\
E(x)=\displaystyle{\frac{1}{2}\left(a+2b+4 c+8 d \right)}, \\
\lambda_{n}={2n (a n-2 (a+b))} \text{ for } \vert a \vert + \vert b \vert \neq 0.
\end{cases}
\end{equation*}
Hence the following difference equation appears
\begin{multline}\label{eq:16}
(2x+1)(a x^{3}+b x^{2}+c x+d) \Delta \nabla \phi_{n}(x) -2 x \left(x^2 (a+2 b)+c+2 d\right) \Delta \phi_{n}(x) \\
 + \left( 2n (a n-2 (a+b)) \left( \frac{1}{4}- x^2 \right) + \sigma_{n} \left( \frac{a}{2}+b+2 c+4 d \right) \right) \phi_{n}(x) =0.
\end{multline}

Since the polynomial solution of equation (\ref{eq:16}) is symmetric, we use the notation
\begin{equation*}
\phi_{n}(x)=\phin{n}{a}{b}{c}{d}{x},
\end{equation*}
for mathematical formulae and $S_{n}(a,b,c,d;x)$ in the sequel. This means that from now we deal with just one characteristic vector  $\vec{V}=(a,b,c,d)$ for any given sub-case.

If the polynomial sequence $\{S_{n}(a,b,c,d;x)\}_{n \geq 0}$ satisfies a three term recurrence relation of type (\ref{eq:30}) then by referring to (\ref{eq:27new}) and (\ref{eq:29new}) we obtain 
\begin{proposition}
The coefficient $\gamma_{n}$ in the three term recurrence relation (\ref{eq:30}) is a rational expression in $n$ where the numerator is a polynomial of degree 4 and the denominator is a polynomial of degree 2, given by
\begin{equation}\label{eq:gmn}
\gamma_{n}=\gamma_{n}\begin{pmatrix} a & b \\ c & d \end{pmatrix}  =\frac{\sum_{i=0}^{4} K_{i}(a,b,c,d)n^{i}}{32 (b-a (n-2)) (b-a (n-1))},
\end{equation}
where
\begin{align*}
K_{4}(a,b,c,d)&=-2 a^2 , \\
K_{3}(a,b,c,d)&= 4 a (3 a+2 b), \\
K_{2}(a,b,c,d)&=-8 \left(3 a^2+a (4 b+c)+b^2\right) , \\
K_{1}(a,b,c,d)&=2 (3 a+2 b) (3 a+4 (b+c))-2 a (-1)^n (a+2 b+4 c+8 d) , \\
K_{0}(a,b,c,d)&= \left((-1)^n-1\right) (3 a+2 b) (a+2 b+4 c+8 d).
\end{align*}
\end{proposition}

For $n=2m$ and $n=2m+1$, $\gamma_n$ in (\ref{eq:gmn}) becomes
\begin{equation*}
\gamma _{2m}  = \frac{m \left(-a^2 (m-1)^3+a \left(2 b (m-1)^2+c(1-m)-d\right)+b (b(1-m)+c)\right)}{(b-2 a (m-1)) (b - a (2 m - 1))}\,,
\end{equation*}
and
\begin{equation*}
\gamma _{2m + 1}  = \frac{{\left( {a(m -1) - b} \right)\left( { - am^3  + b\,m^2  - c\,m + d} \right)}}{{(2a\,m - (a + b))\,(2a\,m - b)}}\,.
\end{equation*}

\begin{remark}\label{remark:positivity}
Once we have explicitly determined $\gamma_{n}$ in the recurrence relation (\ref{eq:30}), a discussion about the situation of this coefficient is extremely important. For instance, analyzing the location of the zeros of orthogonal polynomials would give rise to a positive definite case when $\gamma _n  > 0\,\,\,(\forall n \in {\mathbf{N}})$, the quasi-definite case when $\gamma_{n} \neq 0$, and weak orthogonality case when $\gamma_{n}=0$ for some values of $n$.
However, this discussion completely depends on the four parameters $a,b,c$ and $d$, because $\gamma_{n}$ is in general a rational expression in $n$. In section \ref{sec:5} we provide factorized representations of $\gamma_{n}$ in two hypergeometric sequences, which would allow us to easily determine the orthogonality situation based on the parameters.
\end{remark}

\begin{theorem}\label{theorem:explicit}
The explicit form of the polynomial $S_{n}(a,b,c,d;x)$ is
\begin{multline}\label{eq:explicit2m}
\phin{n}{a}{b}{c}{d}{x} = x^{\sigma_{n}} \sum_{j=0}^{[n/2]} (-1)^{j} \binom{[n/2]}{j} \\ \times\left( \prod_{i=j}^{[n/2]-1} \frac{a(i+\sigma_{n})^{3}-b(i+\sigma_{n})^{2}+c(i+\sigma_{n})-d}{a(i+[n/2]-\sigma_{n+1})-b} \right) (\sigma_{n}-x)_{j} (\sigma_{n}+x)_{j},
\end{multline}
where $[x]$ denotes the integer part of $x$, $(A)_{n}=A(A+1)\cdots(A+n-1)=\Gamma(A+n)/\Gamma(A)$ for $n \geq 1$ with $(A)_{0}=1$ and $\prod\limits_{i = 0}^{ - 1} {(.)}  = 1$.

Moreover, since $(-x)_{j}(x)_{j}=(-1)^{j} \prod_{k=0}^{j-1}(x^2-k^2)$, for $n=2m$ and $n=2m+1$ (\ref{eq:explicit2m}) becomes
\begin{equation*}
\phin{2m}{a}{b}{c}{d}{x}=P_{m}(x^{2})=\sum_{j=0}^{m}   \binom{m}{j} \left( \prod_{i=j}^{m-1} \frac{ai^{3}-bi^{2}+ci-d}{a(i+m-1)-b} \right) \prod_{k=0}^{j-1}(x^2-k^2), \\
\end{equation*}
and
\begin{multline*}
\phin{2m+1}{a}{b}{c}{d}{x}=xR_{m}(x^2)\\= x \sum_{j=0}^{m}  \binom{m}{j} \left( \prod_{i=j}^{m-1} \frac{a (i+1)^{3}-b (i+1)^{2}+c (i+1)-d}{a(i+m)-b} \right) \prod_{k=1}^{j}(x^2-k^2).
\end{multline*}

\end{theorem}
\begin{proof}
By defining the symmetric basis
\begin{multline}\label{eq:basismoments}
\vartheta_{n}(x)=(-1)^{[n/2]} x^{\sigma_{n}}(\sigma_{n}-x)_{[n/2]} (\sigma_{n}+x)_{[n/2]}\\
=(-1)^{[n/2]} x^{\sigma _n } \prod\limits_{k = 0}^{[n/2] - 1} { \left( (k + \sigma _n )^2 -x^{2} \right)} =(-1)^{n} \vartheta_{n}(-x),
\end{multline}

the following straightforward properties are derived
\begin{align}
\Delta \vartheta_{n}(x)&=n \vartheta_{n-1}(x) + \frac{n(n-1)}{2} \vartheta_{n-2}(x)+\sigma_{n+1} \frac{n(n-1)(n-2)}{4} \vartheta_{n-3}(x), \label{deltavn} \\
x \vartheta_{n}(x)&=\vartheta_{n+1}(x) + \sigma_{n+1} \left( \frac{n}{2} \right)^{2} \vartheta_{n-1}(x), \\
\Delta \nabla \vartheta_{n}(x)&=n(n-1) \vartheta_{n-2}(x). \label{deltanablavn}
\end{align}

The expansion of $S_{n}(a,b,c,d;x)$ in terms of the above basis yields
\begin{multline}\label{eq:expansion}
S_{n}(a,b,c,d;x)=\sum_{j=0}^{[n/2]} c_{j}(n) \vartheta_{2j+\sigma_{n}}(x)
\\=\begin{cases}
\displaystyle{\sum_{j=0}^{m} (-1)^{j} c_{j}(2m) (-x)_{j}(x)_{j}}, & n=2m, \\
\displaystyle{\sum_{j=0}^{m} (-1)^{j} c_{j}(2m+1) x(1-x)_{j}(1+x)_{j}}, & n=2m+1,
\end{cases}
\end{multline}
then by using the Navima algorithm \cite{MR1626526,MR1416849,MR1361274} we can reach a solvable recurrence relation for the connection coefficients $c_{j}(n)$. So, by using the latter properties (\ref{deltavn})--(\ref{deltanablavn}) and substituting (\ref{eq:expansion}) in (\ref{eq:16}) we obtain
\begin{equation*}
0=\sum_{j=0}^{[n/2]} c_{j}(n) \left( E_{j}(n) + \vartheta_{2j+\sigma_{n}+2} + F_{j}(n) \vartheta_{2j+\sigma_{n}} +  G_{j}(n) \vartheta_{2j+\sigma_{n}-2} \right),
\end{equation*}
where
\begin{equation*}
E_{j}(n)=2 \left(2 j+\sigma _n-n\right) \left(a (2 j+n-2)+a \sigma _n-2   b\right),
\end{equation*}
\begin{multline*}
F_{j}(n)=4 a (j-1)^2 j (2 j-1) \sigma _{n-1} \\
+\frac{1}{2} \sigma _{n+2} \left(2 j+\sigma _n+1\right){}^2   \left(2 j+\sigma _n-n\right) \left(a (2 j+n-2)+a \sigma _n-2   b\right) \\
+2 j \sigma _{n+1} \left(a \left(4 j^3-6 j^2-j (n-3)   (n+1)-1\right)+2 b (j (-4 j+n+3)-1)\right)\\
+\frac{1}{2} \left(\sigma _n \left(\sigma _n \left(\sigma _n
   \left(\sigma _n \left(8 a j+a \sigma _n-4 a-2 b\right)+a (24
   (j-1) j+5)+4 b (1-3 j)\right) \right. \right. \right. \\
  \left. \left. \left. +2 a (2 j-1) (8 (j-1) j+1)+8 b (2-3   j) j+4 c\right)+a \left(4 (j-1) j (1-2 j)^2+1\right) \right. \right. \\
  \left. \left. +4 j (b (1-4   (j-1) j)+4 c)-4 c\right)-2 n (a+b)+a n^2 \right. \\
   \left.  -4 j (b+4 (c+d))+8 j^2    (b+2 c)\right),
\end{multline*}
and
\begin{multline*}
G_{j}(n)=(j-1) j (2 j-1) \left(2 (j-1) \sigma _{n-1} \left(2 a (j-1)^2   \sigma _{n-1}+b+2 c\right) \right. \\
    \left. -2 \sigma _{n+1} \left((j-1)^2 (a+2   b) \sigma _{n-1}+c+2 d\right)\right) \\
   +\frac{1}{8} \left(2 j+\sigma _n-1\right) \left(2 j+\sigma_n\right) \left(\sigma _n \left(2 j+\sigma _n-1\right)
   \left(\sigma _n \left(2 \left((j-1) \left(-4 j (a+b) \right. \right. \right. \right. \right. \\
    \left. \left. \left. \left. \left. +4 a   j^2+a\right)+2 c\right)+\sigma _n \left(\sigma _n \left(6 a
   j+a \sigma _n-4 a-2 b\right) \right. \right. \right. \right. \\
   +  \left. \left. \left. \left. (2 j-1) (a (6 j-5)-4   b)\right)\right)     +b (4 j-2)-8 (c+d)+8 c j\right)+8 d\right). \\
\end{multline*}
But since $\vartheta_{n}(x)$ is linearly independent, the coefficients $c_{j}(n)$ would satisfy the relation
\begin{equation*}
E_{j-1}(n)c_{j-1}(n) + F_{j}(n) c_{j}(n) + G_{j+1}(n)c_{j+1}(n)=0,
\end{equation*}
which is explicitly solvable with the initial conditions $c_{m}(n)=0$ for $m>[n/2]$, and $c_{[n/2]}(n)=1$, providing (\ref{eq:explicit2m}).
\end{proof}

\begin{remark}\label{remark1}
By substituting the characteristic vector
\begin{equation*}
(a,b,c,d)=\left(\frac{p^{*}}{q^{*}}h,-\frac{p^{*}+r^{*}}{2q^{*}}h,\frac{1}{h},-\frac{q^{*}+s^{*}}{2q^{*}h} \right),
\end{equation*}
and $x=t/h$ in relation (\ref{eq:explicit2m}), the corresponding polynomial satisfies a difference equation of type (\ref{eq:16}) that formally tends to the differential equation (\ref{eq:240}) as $h \to 0$. Hence, the following limit relation can be directly deduced:
\begin{equation*}
\lim_{h \to 0} \phin{n}{\frac{p^{*}}{q^{*}}h}{-\frac{p^{*}+r^{*}}{2q^{*}}h}{\frac{1}{h}}{-\frac{q^{*}+s^{*}}{2q^{*}h}} {\frac{t}{h}} =\ssstar{n}{r^{*}}{s^{*}}{p^{*}}{q^{*}}{t}.
\end{equation*}
\end{remark}

As the recurrence relation (\ref{eq:30}) is explicitly known, the complete form of the orthogonality relation is
\begin{multline}\label{eq:245}
\sum_{x=-\theta}^{\theta} \ww{}{a}{b}{c}{d}{x} \sss{n}{a}{b}{c}{d}{x}  \sss{m}{a}{b}{c}{d}{x} \\
= \prod_{k=1}^{n} \gamma_{k} \begin{pmatrix} a & b \\ c & d \end{pmatrix}\left( \sum_{x=-\theta}^{\theta}  \ww{}{a}{b}{c}{d}{x}  \right) \delta_{n,m},
\end{multline}
where
\begin{equation*}
\ww{}{a}{b}{c}{d}{x}=\left(\frac{1}{4}-x^{2} \right)W(x),
\end{equation*}
is the original weight function and $W(x)$ satisfies the difference equation
\begin{equation}\label{eq:weight}
\frac{W(x+1)}{W(x)}=\frac{(1/2)-x}{(3/2)+x} \frac{-ax^{3}+bx^{2}-cx+d}{a(x+1)^{3}+b(x+1)^{2}+c(c+1)+d}.
\end{equation}

By noting that $A(x)=(2x+1)(ax^{3}+bx^{2}+cx+d)$ for $\vert a \vert + \vert b \vert \neq 0$, two cases can generally happen for the parameter $a$ in (\ref{eq:weight}), i.e. when $a \neq 0$ and $b$ arbitrary or $a=0$ and $b \neq 0$.\\
In the first case, since any arbitrary polynomial of degree 3 has at least one real root, say $x = p \in {\mathbf{R}}$, the aforementioned $A(x)$ can be decomposed in three different forms, i.e.
\begin{multline}\label{eq:nova39}
 A(x) = \,(2x + 1)\,(x - p)\left( {a\,x^2  + u\,x + v} \right) \\
 = \begin{cases}
 (2x + 1)\,(x - p)\left( {a\,x^2  + u\,x + v} \right), & (u^2  < 4av), \\
 (2x + 1)\,a(x - p)(x - q)(x - r), & (u^2  > 4av), \\
 (2x + 1)\,a(x - p)(x - q)^2 ,& (u^2  = 4av). \\
 \end{cases}
\end{multline}

Similarly, in the second case when $a=0$ and $b \neq 0$, $A(x)$ can be decomposed as
\begin{multline}\label{eq:nova40}
 A(x) = \,(2x + 1)\,\left( {b\,x^2  + c\,x + d} \right) \\
= \begin{cases}
 (2x + 1)\,\left( {b\,x^2  + c\,x + d} \right),&(c^2  < 4bd), \\
 (2x + 1)\,b(x - p)(x - q),&(c^2  > 4bd), \\
 (2x + 1)\,b(x - p)^2 ,&(c^2  = 4bd). \\
 \end{cases}
\end{multline}

For the two sub-cases $A(x) = (2x + 1)\,(x - p)\left( {a\,x^2  + u\,x + v} \right)$  in (\ref{eq:nova39}) and  $A(x) = (2x + 1)\left( {b\,x^2  + c\,x + d} \right)$ in (\ref{eq:nova40}), the difference equations corresponding to (\ref{eq:weight}) respectively take the forms
\begin{equation}\label{eq:re39}
\frac{{W(x + 1)}}{{W(x)}} = \frac{{(1/2) - x}}{{(3/2) + x}}\,\frac{{p + x}}{{p - x - 1}}\frac{{a\,x^2  - u\,x + v}}{{a\,(x + 1)^2  + u\,(x + 1) + v}}\,\,\,\,\,\,\,(u^2  < 4av)\,,
\end{equation}
and
\begin{equation}\label{eq:re40}
\frac{{W(x + 1)}}{{W(x)}} = \frac{{(1/2) - x}}{{(3/2) + x}}\,\frac{{b\,x^2  - c\,x + d}}{{b\,(x + 1)^2  + c\,(x + 1) + d}}\,\,\,\,\,\,\,(c^2  < 4bd)\,.
\end{equation}

Since the denominators of the two fractions (\ref{eq:re39}) and (\ref{eq:re40})  are not decomposable in $\mathbf{R}$, we shall deal with them in a separate work [under preparation]. 
However, it is important to note that the cases analyzed in this paper allow us to recover all classical symmetric orthogonal polynomials of a discrete variable (section \ref{section5}).
Hence, let us consider the rest of cases. For the second sub-case of (\ref{eq:nova39}), without loss of generality take $a=1$ and then  $A(x)=(2x+1)(x-p)(x-q)(x-r)$ where $p,q,r \in {\mathbf{R}}$, which indeed covers the third sub-case too. For the second sub-case of (\ref{eq:nova40}) we can similarly consider $A(x)=(2x+1)(x-p)(x-q)$ where $p,q \in {\mathbf{R}}$. This means that there exist only two orthogonal sequences of  $S_{n}(a,b,c,d;x)$ when $A(x)$ is decomposable. Moreover, there exist 20 different weight functions on 32 supports for these two sequences, as the weight functions are not in general unique leading to an indeterminate moment problem. In summary, when the polynomial $A(x)$ is of degree four, a three parametric family (with 14 subcases) appears and when $A(x)$ is of degree three, a two parametric family (with 6 subcases) would appear. Finally, when $A(x)$ is of degree 2, the well known classical symmetric discrete families appear.

\section{Orthogonality supports}

 The question is now how to determine the restrictions of the parameter $\theta$ in the orthogonality support $[-\theta ,\theta ]$ in (\ref{eq:245})? To answer, we should reconsider the main difference equation (\ref{eq:16}) on $[\alpha ,\beta -1] = [-\theta ,\theta ]$ and write it in a self-adjoint form to eventually obtain
\begin{multline}\label{eq:rem3}
\sum_{x=-\theta}^{\theta} \Delta \left( A(x) W(x) \left( \phi_{m}(x) \nabla \phi_{n}(x) - \phi_{n}(x) \nabla \phi_{m}(x) \right) \right) \\ + (\lambda_{n}-\lambda_{m}) \sum_{x=-\theta}^{\theta} \left(\frac{1}{4}-x^{2} \right) W(x) \phi_{n}(x)\phi_{m}(x) \\
+\frac{(-1)^{m}-(-1)^{n}}{2} \left( \frac{a}{2}+b+2c+4d \right) \sum_{x=-\theta}^{\theta} W(x) \phi_{n}(x) \phi_{m}(x)=0.
\end{multline}
On the other side, the identity
\begin{equation*}
\phi_{m}(x) \nabla \phi_{n}(x)-\phi_{n}(x)\nabla \phi_{m}(x)  = \phi_{n}(x) \phi_{m}(x-1) - \phi_{m}(x) \phi_{n}(x-1),
\end{equation*}
simplifies the first sum of (\ref{eq:rem3}) as
\begin{multline}\label{eq:rem5}
\sum_{x=-\theta}^{\theta} \Delta \left( A(x) W(x) \left( \phi_{m}(x) \nabla \phi_{n}(x) - \phi_{n}(x) \nabla \phi_{m}(x) \right) \right) \\
= \left. A(x)W(x) \left( \phi_{n}(x) \phi_{m}(x-1) - \frac{}{} \phi_{m}(x) \phi_{n}(x-1) \right) \right \vert_{x=-\theta}^{x=\theta+1} \\
=A(\theta+1) W(\theta+1) \left(\phi_{n}(\theta+1) \phi_{m}(\theta) - \phi_{m}(\theta+1) \phi_{n}(\theta) \right) \\-
A(-\theta)W(-\theta) \left( \phi_{n}(-\theta) \phi_{m}(-\theta-1) - \phi_{m}(-\theta) \phi_{n}(-\theta-1) \right).
\end{multline}

By taking into account that all weight functions are even, i.e. $W(-x) =W(x)$, the
polynomials are symmetric, i.e. $\phi_{n}(x)=(-1)^{n} \phi_{n}(-x)$, and the Pearson difference equation (\ref{eq:19}) is also valid for $x=\theta$, i.e. $A(\theta+1)W(\theta+1)=A(-\theta)W(\theta)$, relation (\ref{eq:rem5}) would be finally simplified as
\begin{multline}\label{eq:rem7}
\sum_{x=-\theta}^{\theta} \Delta \left( A(x) W(x) \left( \phi_{m}(x) \nabla \phi_{n}(x) - \phi_{n}(x) \nabla \phi_{m}(x) \right) \right) \\
=A(-\theta)W(\theta) \left(1+(-1)^{n+m} \right) \left( \phi_{m}(\theta) \phi_{n}(\theta+1) - \phi_{n}(\theta) \phi_{m}(\theta+1) \right).
\end{multline}

Since $\phi_{m}(\theta) \phi_{n}(\theta+1) - \phi_{n}(\theta) \phi_{m}(\theta+1) \neq 0$, two cases can in general happen for the right hand side of (\ref{eq:rem7}):
\begin{itemize}
\vspace*{0.15cm} \item[i)] If $n + m$ is odd then $1+ (-1)^{n+m} = 0$ and (\ref{eq:rem7}) is automatically zero. However, this case is clear as the sum of any odd summand on a symmetric counter set is equal to
zero.

\item[ii)] If simultaneously $A(-\theta ) = 0$ and $W(\theta) \neq 0$, then (\ref{eq:rem7}) is again equal to zero. This condition is in fact the key point for finding the orthogonality supports.
\end{itemize}

\section{Two hypergeometric orthogonal sequences of $S_{n}(a,b,c,d;x)$}\label{sec:5}

In this section, we introduce two hypergeometric sequences of symmetric orthogonal polynomials, which are particular cases of $S_n(a,b,c,d;x)$ corresponding to two aforementioned cases $a\neq 0$ and $a=0$, respectively and then obtain all possible weight functions together with orthogonality supports for these two sequences.

\subsection{First sequence} If the characteristic vector
\begin{equation*}
(a,b,c,d)=(1,-(p+q+r),pq+pr+qr,-pqr),
\end{equation*}
for $p,q,r \in {\mathbf{R}}$, is replaced in (\ref{eq:16}), then
\begin{multline*}
(2x+1)(x-p)(x-q)(x-r) \Delta \nabla \phi_{n}(x) \\ -2 x \left( x^2 (1-2p-2q-2r)+p q+p r+q r-2 p q r \right) \Delta \phi_{n}(x) \\
 + \left( 2 n (n+2 (p+q+r-1)) \left( \frac{1}{4}- x^2 \right) \right. \\ \left. - \frac{\sigma_{n}}{2}  (2 p-1) (2 q-1) (2 r-1) \right) \phi_{n}(x) =0,
\end{multline*}
has a polynomial solution, which can be represented in terms of hypergeometric series as
\begin{multline}\label{eq:34new}
\sss{n}{1}{-(p+q+r)}{pq+pr+qr}{-pqr}{x}
\\= \frac{(p+\sigma_{n})_{\text{}\left[{n}/{2}\right]} (q+\sigma_{n})_{\text{}\left[{n}/{2}\right]} (r+\sigma_{n})_{\text{}\left[{n}/{2}\right]}}
{\left([n/2]+p+q+r-1+\sigma_{n}\right)_{\text{}\left[{n}/{2}\right]}} \\ \times \ x^{\sigma_{n}} \hyper{4}{3}{-[n/2], \,\, [n/2]+p+q+r-1+\sigma_{n},\,\,\sigma_{n}-x,\,\,\sigma_{n}+x}{p+\sigma_{n},\,\,q+\sigma_{n},\,\,r+\sigma_{n}}{1},
\end{multline}
where $\sigma_{n}$ is defined in (\ref{eq:sigman}) and $_{4} F_3$ is the well known hypergeometric function of order $(4,3)$.
Moreover, it satisfies the recurrence relation
\begin{multline}\label{eq:444}
\phi_{n+1}(x)=x \phi_{n}(x) - \gamma_{n}\begin{pmatrix} {1} & {-(p+q+r)} \\ {pq+pr+qr} & {-pqr} &\end{pmatrix} \phi_{n-1}(x), \\ (\phi_{0}(x)=1, \quad \phi_{1}(x)=x),
\end{multline}
where
\begin{multline}\label{eq:gamman}
\gamma_{n}\begin{pmatrix}  {1} & {-(p+q+r)} \\ {pq+pr+qr} & {-pqr}\end{pmatrix}= \left\{ -2 n^{4} -4 (-3 + 2 p + 2 q + 2 r) n^{3} \right. \\ \left. -8 \left(p^2+p (3 q+3 r-4)+3 q
   r+(q-4) q+r^2-4 r+3\right) n^{2}  \right. \\ \left. + \left( 2 (-1)^n (-1 + 2 p) (-1 + 2 q) (-1 + 2 r)  \right. \right. \\ \left. \left. -
 2 (-3 + 2 p + 2 q + 2 r) (3 + 4 q (-1 + r) - 4 r + 4 p (-1 + q + r)) \right) n \right. \\ \left. + (-1 + (-1)^n) (-1 + 2 p) (-1 + 2 q) (-1 + 2 r) (-3 + 2 p + 2 q + 2 r) \right\}\\ /\left\{32 (n+p+q+r-2) (n+p+q+r-1) \right\},
\end{multline}
dividing to
\[
\gamma_{2n}=-\frac{n (n+p+q-1) (n+p+r-1) (n+q+r-1)}{(2 n+p+q+r-2) (2  n+p+q+r-1)}, \\
\]
and
\[
\gamma_{2n+1}=-\frac{(n+p) (n+q) (n+r) (n+p+q+r-1)}{(2 n+p+q+r-1) (2 n+p+q+r)}.
\]
The latter representations would allow us to analyze the sign of $\gamma_{n}$ in terms of the values of $p$, $q$ and $r$ as explained in Remark \ref{remark:positivity}.

By noting the relations (\ref{eq:444}) and (\ref{eq:gamman}), the orthogonality relation of the first sequence is
\begin{multline*}
\sum_{x=-\theta}^{\theta} \left( \ww{}{1}{-(p+q+r)}{pq+pr+qr}{-pqr}{x} S_{n}(x)S_{m}(x) \right) \\
=\prod_{k=1}^{n} \gamma_{k} \begin{pmatrix} {1} & {-(p+q+r)} \\ {pq+pr+qr} & {-pqr}\end{pmatrix} \\ \times\left(\sum_{x=-\theta}^{\theta} \ww{}{1}{-(p+q+r)}{pq+pr+qr}{-pqr}{x} \right) \delta_{n,m},
\end{multline*}
in which
\begin{equation*}
S_{n}(x)=\sss{n}{1}{-(p+q+r)}{pq+pr+qr}{-pqr}{x},
\end{equation*}
and, finally
\begin{equation*}
\ww{}{1}{-(p+q+r)}{pq+pr+qr}{-pqr}{x}=\left( \frac{1}{4}-x^{2} \right) W(x),
\end{equation*}
denotes the original weight function corresponding to the hypergeometric polynomialdenotes the original weight function corresponding to the hypergeometric polynomial (\ref{eq:34new}). As a consequence, we obtain
\begin{proposition}
The function $W(x)$ satisfies a particular case of the difference equation (\ref{eq:weight}) 
\begin{equation}\label{eq:weightA1}
\frac{W(x+1)}{W(x)}=\frac{-x+(1/2)}{x+(3/2)} \frac{-x-p}{x+1-p} \frac{-x-q}{x+1-q} \frac{-x-r}{x+1-r},
\end{equation}
giving rise to the 16 symmetric solutions of (\ref{eq:weightA1}) listed below
\begin{multline}\label{eq:w1}
W_{1}(x;p,q,r)=\left( \Gamma(1-p+x) \Gamma(1-p-x) \Gamma(1-q+x) \Gamma(1-q-x) \right. \\
\left. \times \Gamma(1-r+x) \Gamma(1-r-x) \Gamma(3/2+x) \Gamma(3/2-x)\right)^{-1},
\end{multline}
\begin{multline}
W_{2,1}(x;p,q,r)\\=\frac{\Gamma(p+x) \Gamma(p-x)}{\Gamma(1-q+x)\Gamma(1-q-x) \Gamma(1-r+x) \Gamma(1-r-x) \Gamma(3/2+x) \Gamma(3/2-x)},
\end{multline}
\begin{multline}
W_{2,2}(x;p,q,r)=W_{2,1}(x;q,p,r)\\=\frac{\Gamma(q+x) \Gamma(q-x)}{\Gamma(1-p+x)\Gamma(1-p-x) \Gamma(1-r+x) \Gamma(1-r-x) \Gamma(3/2+x) \Gamma(3/2-x)},
\end{multline}
\begin{multline}
W_{2,3}(x;p,q,r)=W_{2,1}(x;r,q,p)\\=\frac{\Gamma(r+x) \Gamma(r-x)}{\Gamma(1-p+x)\Gamma(1-p-x) \Gamma(1-q+x) \Gamma(1-q-x) \Gamma(3/2+x) \Gamma(3/2-x)},
\end{multline}
\begin{equation}
W_{3,1}(x;p,q,r)=\frac{\Gamma(p+x) \Gamma(p-x)\Gamma(q+x)\Gamma(q-x)}{ \Gamma(1-r+x) \Gamma(1-r-x) \Gamma(3/2+x) \Gamma(3/2-x)},
\end{equation}
\begin{equation}
W_{3,2}(x;p,q,r)=W_{3,1}(x;p,r,q)=\frac{\Gamma(p+x) \Gamma(p-x)\Gamma(r+x)\Gamma(r-x)}{ \Gamma(1-q+x) \Gamma(1-q-x) \Gamma(3/2+x) \Gamma(3/2-x)},
\end{equation}
\begin{equation}
W_{3,3}(x;p,q,r)=W_{3,2}(x;r,q,p)=\frac{\Gamma(q+x) \Gamma(q-x)\Gamma(r+x)\Gamma(r-x)}{ \Gamma(1-p+x) \Gamma(1-p-x) \Gamma(3/2+x) \Gamma(3/2-x)},
\end{equation}
\begin{equation}
W_{4,1}(x;p,q,r)=\frac{\Gamma(p+x) \Gamma(p-x)\Gamma(-1/2+x) \Gamma(-1/2-x)}{ \Gamma(1-q+x) \Gamma(1-q-x) \Gamma(1-r+x)\Gamma(1-r-x)},
\end{equation}
\begin{multline}
W_{4,2}(x;p,q,r)=W_{4,1}(x;q,p,r)\\=\frac{\Gamma(q+x) \Gamma(q-x)\Gamma(-1/2+x) \Gamma(-1/2-x)}{ \Gamma(1-p+x) \Gamma(1-p-x) \Gamma(1-r+x)\Gamma(1-r-x)},
\end{multline}
\begin{multline}
W_{4,3}(x;p,q,r)=W_{4,1}(x;r,q,p)\\=\frac{\Gamma(r+x) \Gamma(r-x)\Gamma(-1/2+x) \Gamma(-1/2-x)}{ \Gamma(1-p+x) \Gamma(1-p-x) \Gamma(1-q+x)\Gamma(1-q-x)},
\end{multline}
\begin{equation}
W_{5,1}(x;p,q,r)=\frac{\Gamma(p+x) \Gamma(p-x)\Gamma(q+x) \Gamma(q-x)\Gamma(-1/2+x) \Gamma(-1/2-x)}{  \Gamma(1-r+x)\Gamma(1-r-x)},
\end{equation}
\begin{multline}
W_{5,2}(x;p,q,r)=W_{5,1}(x;p,r,q)\\=\frac{\Gamma(p+x) \Gamma(p-x)\Gamma(r+x) \Gamma(r-x)\Gamma(-1/2+x) \Gamma(-1/2-x)}{  \Gamma(1-q+x)\Gamma(1-q-x)},
\end{multline}
\begin{multline}
W_{5,3}(x;p,q,r)=W_{5,1}(x;r,q,p)\\=\frac{\Gamma(q+x) \Gamma(q-x)\Gamma(r+x) \Gamma(r-x)\Gamma(-1/2+x) \Gamma(-1/2-x)}{  \Gamma(1-p+x)\Gamma(1-p-x)},
\end{multline}
\begin{multline}
W_{6}(x;p,q,r)\\=\frac{\Gamma(-1/2+x) \Gamma(-1/2-x)}{\Gamma(1-p+x) \Gamma(1-p-x) \Gamma(1-q+x)\Gamma(1-q-x)  \Gamma(1-r+x)\Gamma(1-r-x)},
\end{multline}
\begin{equation}
W_{7}(x;p,q,r)=\frac{\Gamma(p+x) \Gamma(p-x)\Gamma(q+x) \Gamma(q-x)\Gamma(r+x) \Gamma(r-x)}{  \Gamma(3/2+x) \Gamma(3/2-x)},
\end{equation}
\begin{multline}\label{eq:w8}
W_{8}(x;p,q,r)=\Gamma(p+x) \Gamma(p-x)\Gamma(q+x) \Gamma(q-x) \\
\times \Gamma(r+x) \Gamma(r-x) \Gamma(-1/2+x) \Gamma(-1/2-x).
\end{multline}
\end{proposition}

As the original weight functions corresponding to all above 16 cases are as
\[
\left( \frac{1}{4}-x^{2} \right) W_{k}(x),
\]
the two following identities are remarkable in this direction
\begin{equation*}
(1/4-x^{2})\Gamma(-1/2+x)\Gamma(-1/2-x)=\Gamma(1/2+x)\Gamma(1/2-x),
\end{equation*}
and
\begin{equation*}
\frac{1/4-x^{2}}{\Gamma(3/2+x)\Gamma(3/2-x)}=\frac{1}{\Gamma(1/2+x)\Gamma(1/2-x)}.
\end{equation*}

For example, for the last given case the original weight function becomes
\begin{multline*}
\ww{8}{1}{-(p+q+r)}{pq+pr+qr}{-pqr}{x} =\left( \frac{1}{4}-x^{2} \right) W_{8}(x;p,q,r) \\
=\Gamma(p+x) \Gamma(p-x)\Gamma(q+x) \Gamma(q-x) \Gamma(r+x) \Gamma(r-x) \Gamma\left(\frac{1}{2}+x\right) \Gamma\left(\frac{1}{2}-x\right).
\end{multline*}

By noting section 3, the following table now shows the orthogonality supports of each weight function given in (\ref{eq:w1}) to (\ref{eq:w8}) together with their parameter restrictions in which ${\mathbf{Z}}^{-}=\{0,-1,-2,\dots\}$.

\begin{table}[ht]
\caption{Orthogonality supports for the first hypergeometric sequence} \label{table1} \centering 
\begin{tabular}{lll} 
\hline
$W_{k}(x)$ & Support & Parameter restrictions  \\ \hline & & \\
$W_{1}(x;p,q,r)$ & $[-p,p]$ & $p\in {\mathbf{Z}}^{-}$, $1-q\pm p \not \in {\mathbf{Z}}^{-}$, $1-r\pm p \not \in {\mathbf{Z}}^{-}$. \\
                   & $[-q,q]$ & $q\in {\mathbf{Z}}^{-}$, $1-p\pm q \not \in {\mathbf{Z}}^{-}$, $1-r\pm q \not \in {\mathbf{Z}}^{-}$. \\
                   & $[-r,r]$ & $r\in {\mathbf{Z}}^{-}$, $1-q\pm r \not \in {\mathbf{Z}}^{-}$, $1-p\pm r  \not \in {\mathbf{Z}}^{-}$. \\
& & \\
$W_{2,1}(x;p,q,r)$   & $[-q,q]$ & $q\in {\mathbf{Z}}^{-}$, $p\pm q \not \in {\mathbf{Z}}^{-}$, $1-r\pm q \not \in {\mathbf{Z}}^{-}$.\\
                      & $[-r,r]$ & $r\in {\mathbf{Z}}^{-}$, $p\pm r \not \in {\mathbf{Z}}^{-}$, $1-q\pm r \not \in {\mathbf{Z}}^{-}$.\\
                      & & \\
$W_{3,1}(x;p,q,r)$ & $[-r,r]$ & $r\in {\mathbf{Z}}^{-}$, $p\pm r \not \in {\mathbf{Z}}^{-}$, $q \pm r \not \in {\mathbf{Z}}^{-}$. \\
& & \\
$W_{4,1}(x;p,q,r)$ & $[-q,q]$ & $q\in {\mathbf{Z}}^{-}$, $p\pm q \not \in {\mathbf{Z}}^{-}$, $1-r\pm q \not \in {\mathbf{Z}}^{-}$. \\
& $[-r,r]$ &   $r\in {\mathbf{Z}}^{-}$, $p\pm r \not \in {\mathbf{Z}}^{-}$, $1-q\pm r \not \in {\mathbf{Z}}^{-}$. \\
& & \\
$W_{5,1}(x;p,q,r)$ & $[-r,r]$ &  $r\in {\mathbf{Z}}^{-}$, $p\pm r \not \in {\mathbf{Z}}^{-}$, $q \pm r \not \in {\mathbf{Z}}^{-}$. \\
& & \\
$W_{6}(x;p,q,r)$ & $[-p,p]$ &  $p\in {\mathbf{Z}}^{-}$,  $1-q\pm p \not \in {\mathbf{Z}}^{-}$, $1-r\pm q \not \in {\mathbf{Z}}^{-}$. \\
& $[-q,q]$ & $q\in {\mathbf{Z}}^{-}$, $1-r\pm q \not \in {\mathbf{Z}}^{-}$, $1-p\pm q \not \in {\mathbf{Z}}^{-}$. \\
& $[-r,r]$ &  $r\in {\mathbf{Z}}^{-}$, $1-q\pm r \not \in {\mathbf{Z}}^{-}$, $1-p\pm r \not \in {\mathbf{Z}}^{-}$. \\
& & \\
$W_{7}(x;p,q,r)$ & --- & --- \\
$W_{8}(x;p,q,r)$ & --- & --- \\
&& \\ \hline
\end{tabular}
\end{table}

\begin{remark}\label{remark2}
Since some weight functions are symmetric with respect to the parameters $p$, $q$ and $r$, e.g. $W_{4,2}(x;p,q,r)=W_{4,1}(x;q,p,r)$ and $W_{4,3}(x;p,q,r)=W_{4,1}(x;r,q,p)$, their orthogonality supports and parameter restrictions can be directly derived via the rows of table \ref{table1} by just interchanging the parameters. Also note that $W_{7} (x; p, q, r)$ and $W_{8} (x; p, q, r)$ have no valid orthogonality support. Therefore, there are totally $24$ eligible orthogonality supports for the first sequence.
\end{remark}

\begin{remark}
Let the weight functions corresponding to the first sequence be indicated as $\varrho_{i}(x)=(1/4-x^{2})W_{i}(x)$. Then they satisfy the difference equation
\begin{equation*}
\frac{\varrho_{i}(x+1)}{\varrho_{i}(x)} = \frac{(p+x) (q+x) (r+x)}{(-p+x+1) (-q+x+1) (-r+x+1)},
\end{equation*}
which is equivalent to
\begin{equation*}
\Delta (M_{1}(x) \varrho_{i}(x)) = N_{1}(x) \varrho_{i}(x),
\end{equation*}
where
\begin{equation*}
M_{1}(x)=(x-p) (x-q) (x-r), \quad N_{1}(x)=2 x^2 (p+q+r)+2 p q r.
\end{equation*}
Thus, the hypergeometric polynomials (\ref{eq:34new}) constitute a $\Delta$-semiclassical sequence of class one \cite{1181.42028}, for which we have explicitly given their orthogonality condition, three term recurrence relation, hypergeometric representation and in the sequel in section \ref{section:moments} we will compute the factorial moments with respect to the basis $\vartheta_{n}(x)$ defined in (\ref{eq:basismoments}).
\end{remark}

\subsubsection{A numerical example for the first sequence}

Let us replace $p=-9$, $q=-10$ and $r=-11$ in (\ref{eq:34new}) to get
\begin{multline}\label{pol-example}
\sss{n}{1}{30}{299}{990}{x} \\
= \frac{(-9+\sigma_{n})_{\text{}\left[{n}/{2}\right]} (-10+\sigma_{n})_{\text{}\left[{n}/{2}\right]} (-11+\sigma_{n})_{\text{}\left[{n}/{2}\right]}}
{\left(-31+\text{}\left[{n}/{2}\right]+\sigma_{n}\right)_{\text{}\left[{n}/{2}\right]}} \\
\times x^{\sigma_{n}} \hyper{4}{3}{-[n/2], \,\, -31+[n/2]+\sigma_{n},\,\,\sigma_{n}-x,\,\,\sigma_{n}+x}
{-9+\sigma_{n},\,\,-10+\sigma_{n},\,\,-11+\sigma_{n}}{1}.
\end{multline}

Since $p \in {\mathbf{Z}}^{-}$, $1-q\pm p \not \in {\mathbf{Z}}^{-}$ and $1-r\pm p \not \in {\mathbf{Z}}^{-}$, referring to Table \ref{table1} shows that there are two possible weight functions for the selected parameters which are orthogonal with respect to the sequence $ S_{n}(1,30,299,990;x)$ on the support $\{-9,-8,\dots,8,9\}$, i.e.
\begin{multline*}
\ww{1}{1}{30}{299}{990}{x}=\left(\frac{1}{4}-x^{2} \right) W_{1}(x;-9,-10,-11) \\
=\left( \Gamma(10+x) \Gamma(10-x) \Gamma(11+x) \Gamma(11-x) \right. \\ \left. \Gamma(12+x) \Gamma(12-x) \Gamma(1/2+x) \Gamma(1/2-x) \right)^{-1},
\end{multline*}
and
\begin{multline*}
\ww{6}{1}{30}{299}{990}{x}=\left(\frac{1}{4}-x^{2} \right) W_{6}(x;-9,-10,-11) \\
=\frac{\Gamma(1/2+x) \Gamma(1/2-x)}{ \Gamma(10+x) \Gamma(10-x) \Gamma(11+x) \Gamma(11-x) \Gamma(12+x) \Gamma(12-x)}.
\end{multline*}
Hence there are two orthogonality relations corresponding to the polynomials (\ref{pol-example}), respectively as follows
\begin{multline*}
\sum_{x=-9}^{9}\ww{i}{1}{30}{299}{990}{x} S_{n}(1,30,299,990;x)S_{m}(1,30,299,990;x) \\
=\alpha_{i} \frac{36183421612800000 (-1)^n (-29)_{\left\lfloor
   \frac{n-1}{2}\right\rfloor -1} (-19)_{\left\lfloor
   \frac{n}{2}\right\rfloor -2} (-18)_{\left\lfloor
   \frac{n}{2}\right\rfloor -2} (-17)_{\left\lfloor
   \frac{n}{2}\right\rfloor -2} }{(-31)_n (-30)_n} \\ \times (-9)_{\left\lfloor
   \frac{n-1}{2}\right\rfloor -1} (-8)_{\left\lfloor
   \frac{n-1}{2}\right\rfloor -1} (-7)_{\left\lfloor
   \frac{n-1}{2}\right\rfloor -1} \Gamma \left(\left\lfloor
   \frac{n}{2}\right\rfloor +1\right) \delta_{n,m},
\end{multline*}
where $\left\lfloor x \right\rfloor$ denotes the floor function, and
\[
\alpha_{i}=\sum_{x=-9}^{9}\ww{i}{1}{30}{299}{990}{x}
=\begin{cases} {\beta}/{ \pi }, & i=1, \\ \beta \pi , & i=6, \end{cases}
\]
in which
\[
\beta=\frac{667}{1998530094928466929986605067918114816000000000 }.
\]

\subsection{Second sequence}
If the characteristic vector
\begin{equation*}
(a,b,c,d)=(0,1,-p-q,pq),
\end{equation*}
for $p, q \in {\mathbf{R}}$, is replaced in the difference equation (\ref{eq:16}), then
\begin{multline*}
(2x+1)(x-p) (x-q) \Delta \nabla \phi_{n}(x) -2 x \left(2x^2 +p (2 q-1)-q \right) \Delta \phi_{n}(x) \\
 + \left( -4n \left( \frac{1}{4}- x^2 \right) + \frac{1-(-1)^{n}}{2} (2 p-1) (2 q-1) \right) \phi_{n}(x) =0,
\end{multline*}
has a basis solution, which can be written in terms of hypergeometric series
\begin{multline}\label{eq:fourthfinal}
\sss{n}{0}{1}{-p-q}{pq}{x}  = (p+\sigma_{n})_{\text{}\left[{n}/{2}\right]} (q+\sigma_{n})_{\text{}\left[{n}/{2}\right]}
\\ \times x^{\sigma_{n}} \hyper{3}{2}{-[n/2], \,\,\sigma_{n}-x,\,\,\sigma_{n}+x}{p+\sigma_{n},\,\,q+\sigma_{n}}{1}.
\end{multline}
The above polynomial is a limit case of the polynomial (\ref{eq:34new}) when $r \to \infty$. Moreover, it satisfies a recurrence relation of type (\ref{eq:30}) with
\begin{equation*}
\gamma_{n}\begin{pmatrix} {0} & {1} \\ {-p-q} & {pq}  \end{pmatrix} = \frac{1}{8} \left(-2 n^2-4 n   (p+q-1)+\left((-1)^n-1\right   ) (2 p-1) (2 q-1)\right),
\end{equation*}
which implies
\[
\gamma_{2n}=-n (-1 + n + p + q),
\]
and
\[
\gamma_{2n+1}=-(n + p) (n + q).
\]
Hence, the orthogonality relation corresponding to the second sequence takes the form
\begin{multline*}
\sum_{x=-\theta}^{\theta} \left( \ww{}{0}{1}{-p-q}{pq}{x} \right. \\ \times \left. \sss{n}{0}{1}{-p-q}{pq}{x} \sss{m}{0}{1}{-p-q}{pq}{x}  \right) \\
=\prod_{k=1}^{n} \gamma_{k} \begin{pmatrix} {0} & {1} \\ {-p-q} & {pq} \end{pmatrix} \left(\sum_{x=-\theta}^{\theta} \ww{}{0}{1}{-p-q}{pq}{x} \right) \delta_{n,m},
\end{multline*}
where
\begin{equation*}
\ww{}{0}{1}{-p-q}{pq}{x}=\left( \frac{1}{4}-x^{2} \right) W(x),
\end{equation*}
is the original weight function. Thus, we obtain
\begin{proposition}
The function $W(x)$ satisfies a particular case of the difference equation (\ref{eq:weight}) as
\begin{equation}\label{eq:weightA4}
\frac{W(x+1)}{W(x)}=\frac{1/2-x}{3/2+x} \frac{-x-p}{x+1-p} \frac{-x-q}{x+1-q},
\end{equation}
giving rise to 8 symmetric solutions for equation (\ref{eq:weightA4}) respectively as follows:
\begin{multline}
W_{9}(x;p,q)\\=\frac{1}{\Gamma(1-p+x) \Gamma(1-p-x) \Gamma(1-q+x) \Gamma(1-q-x)\Gamma(3/2+x) \Gamma(3/2-x)},
\end{multline}
\begin{equation}\label{eq:ww291}
W_{10,1}(x;p,q)=\frac{\Gamma(p+x)\Gamma(p-x)}{ \Gamma(1-q+x) \Gamma(1-q-x)\Gamma(3/2+x) \Gamma(3/2-x)},
\end{equation}
\begin{equation}\label{eq:ww292}
W_{10,2}(x;p,q)=\frac{\Gamma(q+x)\Gamma(q-x)}{ \Gamma(1-p+x) \Gamma(1-p-x)\Gamma(3/2+x) \Gamma(3/2-x)},
\end{equation}
\begin{equation}
W_{11}(x;p,q)=\frac{\Gamma(p+x)\Gamma(p-x)\Gamma(q+x)\Gamma(q-x)}{ \Gamma(3/2+x) \Gamma(3/2-x)},
\end{equation}
\begin{equation}
W_{12}(x;p,q)=\frac{\Gamma(-1/2+x) \Gamma(-1/2-x)}{\Gamma(1-p+x) \Gamma(1-p-x) \Gamma(1-q+x) \Gamma(1-q-x)},
\end{equation}
\begin{equation}
W_{13,1}(x;p,q)=\frac{\Gamma(p+x) \Gamma(p-x)\Gamma(-1/2+x) \Gamma(-1/2-x)}{ \Gamma(1-q+x) \Gamma(1-q-x)},
\end{equation}
\begin{equation}
W_{13,2}(x;p,q)=\frac{\Gamma(q+x) \Gamma(q-x)\Gamma(-1/2+x) \Gamma(-1/2-x)}{ \Gamma(1-p+x) \Gamma(1-p-x)},
\end{equation}
\begin{equation}
W_{14}(x;p,q)=\Gamma(p+x) \Gamma(p-x)\Gamma(q+x) \Gamma(q-x)\Gamma(-1/2+x) \Gamma(-1/2-x).
\end{equation}
\end{proposition}

Table \ref{table2} shows the orthogonality supports of each above-mentioned weights and their parameter restrictions.

\begin{table}[ht]
\caption{Orthogonality supports for the second hypergeometric sequence} \label{table2} \centering 
\begin{tabular}{lll} 
\hline
$W_{k}(x)$ & Support & Parameter restrictions  \\ \hline & & \\
$W_{9}(x;p,q)$ & $[-p,p]$ & $p \in {\mathbf{Z}}^{-}$, $1-q\pm p \not \in {\mathbf{Z}}^{-}$. \\
	 & $[-q,q]$ &  $q \in {\mathbf{Z}}^{-}$, $1-p\pm q \not \in {\mathbf{Z}}^{-}$. \\
& & \\
$W_{10,1}(x;p,q)$ & $[-q,q]$ &  $q \in {\mathbf{Z}}^{-}$, $p\pm q \not \in {\mathbf{Z}}^{-}$. \\ \\
& & \\
$W_{11}(x;p,q)$ & --- & --- \\
& & \\
$W_{12}(x;p,q)$ & $[-p,p]$ & $p \in {\mathbf{Z}}^{-}$, $1-q\pm p \not \in {\mathbf{Z}}^{-}$.  \\
	& $[-q,q]$ &  $q \in {\mathbf{Z}}^{-}$, $1-p\pm q \not \in {\mathbf{Z}}^{-}$. \\
& & \\
$W_{13,1}(x;p,q)$ & $[-q,q]$ & $q \in {\mathbf{Z}}^{-}$, $p\pm q \not \in {\mathbf{Z}}^{-}$. \\
& & \\
$W_{14}(x;p,q)$ & --- & --- \\ && \\ \hline
\end{tabular}
\end{table}

\begin{remark}
As the main orthogonality relation (\ref{eq:245}) shows, an important part that one has to compute in norm square value is $\sum_{x=-\theta}^{\theta} W(x;a,b,c,d)$ where $[-\theta ,\theta]$ is the same orthogonality supports as determined in tables \ref{table1} and \ref{table2}. Since the functions $\{W_{k}(x;a,b,c,d)\}_{k=1}^{20}$ are all even, the aforesaid sums can be simplified on its orthogonality supports by using two identities
\[
\Gamma(p+x)=\Gamma(p)(p)_{x} \quad \text{and} \quad \Gamma(p-x)=\frac{\Gamma(p)(-1)^{x}}{(-p)_{x}},
\]
and this fact that
\[
\sum_{x=-\theta}^{\theta} W(x;a,b,c,d)=2 \sum_{x=0}^{\theta} W(x;a,b,c,d)-W(0;a,b,c,d).
\]
For example, for the first given weight function $W_{1}(x;p,q,r)$ on e.g. $[-p,p]$ we have
\begin{multline*}
\sum_{x=-p}^{p} \left(\frac{1}{4}-x^{2} \right) W_{1}(x;p,q,r) \\=\frac{1}{\pi \Gamma^{2}(1-p) \Gamma^{2}(1-q) \Gamma^{2}(1-r)} \left(2\sum_{x=0}^{p} \frac{(p)_{x}(q)_{x}(r)_{x}}{(1-p)_{x}(1-q)_{x}(1-r)_{x}} -1 \right).
\end{multline*}
\end{remark}

\begin{remark}
Let the weight functions corresponding to the second sequence be indicated as $\varrho_{j}(x)=(1/4-x^{2})W_{j}(x)$. Then they satisfy the difference equation
\begin{equation*}
\frac{\varrho_{j}(x+1)}{\varrho_{j}(x)} = -\frac{(p+x) (q+x)}{(-p+x+1) (-q+x+1)},
\end{equation*}
which is equivalent to
\begin{equation*}
\Delta (M_{2} \varrho_{j}(x)) = N_{2} \varrho_{j}(x),
\end{equation*}
where
\begin{equation*}
M_{2}(x)=(x-p) (x-q), \quad N_{2}(x)=-2 \left(p q+x^2\right).
\end{equation*}
\end{remark}

\subsubsection{A numerical example for the second sequence}

Let us replace $p=14$ and $q=-10$ in (\ref{eq:34new}) to get
\begin{multline}\label{pol-example2}
\sss{n}{0}{1}{-4}{-140}{x}  = (14+\sigma_{n})_{\text{}\left[{n}/{2}\right]} (-10+\sigma_{n})_{\text{}\left[{n}/{2}\right]}
\\ \times x^{\sigma_{n}} \hyper{3}{2}{-[n/2], \,\,\sigma_{n}-x,\,\,\sigma_{n}+x}{14+\sigma_{n},\,\,-10+\sigma_{n}}{1}.
\end{multline}

Since $q \in {\mathbf{Z}}^{-}$ and $p\pm q \not \in {\mathbf{Z}}^{-}$, referring to Table \ref{table2} shows that there are two possible weight functions for the selected parameters which are orthogonal with respect to the sequence $\S_{n}(0,1,-4,-140;x)$ on the support $\{-10,-9,\dots,9,10\}$, i.e.
\begin{multline*}
\ww{10,1}{0}{1}{-4}{-140}{x}=\left(\frac{1}{4}-x^{2} \right) W_{10,1}(x;14,-10) \\
=\frac{(13-x) (12-x) (11-x) (x+11) (x+12) (x+13) \cos (\pi  x)}{\pi },
\end{multline*}
and
\begin{multline*}
\ww{13,1}{0}{1}{-4}{-140}{x}=\left(\frac{1}{4}-x^{2} \right) W_{13,1}(x;14,-10) \\
=\pi (13-x) (12-x) (11-x) (x+11) (x+12) (x+13) \sec (\pi  x).
\end{multline*}

Hence there are two orthogonality relations corresponding to the polynomials (\ref{pol-example2}) respectively as follows
\begin{multline*}
\sum_{x=-10}^{10}\ww{j,1}{0}{1}{-4}{-140}{x} S_{n}(0,1,-4,-140;x)S_{m}(0,1,-4,-140;x) \\
=\tilde{\alpha}_{j} \frac{(-1)^{n+1} (-9)_{\left\lfloor \frac{n-1}{2}\right\rfloor } \Gamma
   \left(\left\lfloor \frac{n-1}{2}\right\rfloor +15\right) \Gamma
   \left(\left\lfloor \frac{n}{2}\right\rfloor +1\right) \Gamma
   \left(\left\lfloor \frac{n}{2}\right\rfloor +4\right)}{3736212480},
\end{multline*}
where
\[
\tilde{\alpha}_{j}=\sum_{x=-10}^{10} \ww{j,1}{0}{1}{-4}{-140}{x}=\begin{cases} {10296/\pi },& j=10, \\ 10296 \pi ,& j=13.\end{cases}
\]

\section{Moments of the two introduced sequences}\label{section:moments}

To compute the moments of a continuous distribution, different bases should be considered. For example, in the normal distribution the canonical basis $\{x^{j}\}_{j \geq 0}$ is used to get
\[
\int_{-\infty}^{\infty} x^{n} \frac{e^{-x^2/2}}{\sqrt{2\pi}} \text{d}x= \begin{cases} 0, & n=2m+1, \\ \displaystyle{\frac{2^{m+1/2}}{\sqrt{2\pi }} \Gamma \left(m+\frac{1}{2}\right)}, & n=2m, \end{cases}
\]
while for the Jacobi weight function $(1-x)^{\alpha} (1+x)^{\beta}$ as the shifted beta distribution on $[-1,1]$, using one of the two bases $\{(1-x)^{j}\}_{j}$ or $\{(1+x)^{j}\}_{j}$ is appropriate for this purpose.

This matter similarly holds for the moments of discrete orthogonal polynomials. For instance, in the negative hypergeometric distribution corresponding to Hahn polynomials, it is more convenient to use the Pochhammer basis $\{(-x)_{n}\}_{n \geq 0}$, instead of the canonical basis, to get
\begin{multline*}
\sum_{x=0}^{N-1} \frac{\Gamma(N) \Gamma(\alpha+\beta+2) \Gamma(\alpha+N-x) \Gamma(\beta+x+1)}{\Gamma(\alpha+1) \Gamma(\beta+1) \Gamma(\alpha+\beta+N+1) \Gamma(N-x) \Gamma(x+1)} (-x)_{n} \\
= (-1)^{n}\frac{(1-N)_{n} (\beta+1)_{n}}{(\alpha+\beta+2)_{n}}.
\end{multline*}

Following this approach, for the weight functions $\varrho_{i}(x)$ appeared in the two introduced hypergeometric sequences, we shall compute the moments of the form
\begin{equation*}
(\varrho_{i})_{n} = \sum_{x=-\theta}^{\theta} \vartheta_{n}(x) \varrho_{i}(x) = \sum_{x=-\theta}^{\theta} \vartheta_{n}(x) \left( \frac{1}{4}-x^{2} \right) W_{i}(x).
\end{equation*}
where the basis $ \vartheta_{n}(x) $ is defined in (\ref{eq:basismoments}).

Since $\vartheta_{2n+1}(x)$ are odd polynomials, clearly all odd moments with respect to this basis are zero. Moreover, from (\ref{eq:34new}) and  (\ref{eq:fourthfinal}) and also using the orthogonality property of the polynomials it can be proved by induction that the even moments corresponding to the first and second sequences respectively satisfy the following recurrence  relations
\[
(\varrho_{i})_{2n}= - \frac{(p+n-1)(q+n-1)(r+n-1)}{p+q+r+n-1} (\varrho_{i})_{2n-2},
\]
and
\[
(\varrho_{j})_{2n} = -(n+p-1) (n+q-1) (\varrho_{j})_{2n-2}.
\]
Therefore, if the aforesaid weight functions are normalized with the first moment equal to one, then we eventually obtain
\begin{equation*}
\sum_{x=-\theta}^{\theta} \left(\frac{1}{4}-x^{2} \right) W_{i}(x) \vartheta_{n}(x)=
\begin{cases}
0, & n=2m+1, \\[3mm]
\displaystyle{\frac{(-1)^m (p)_m (q)_m (r)_m}{(p+q+r)_m}}, & n=2m,
\end{cases}
\end{equation*}
for the weight functions of the first sequence, and
\begin{equation*}
\sum_{x=-\theta}^{\theta} \left(\frac{1}{4}-x^{2} \right) W_{j}(x) \vartheta_{n}(x)=
\begin{cases}
0, & n=2m+1, \\[3mm]
(-1)^m (p)_m (q)_m, & n=2m.
\end{cases}
\end{equation*}
for the weight functions of the second sequence.

\section{A particular example of $S_{n}(a,b,c,d;x)$ generating all classical symmetric orthogonal polynomials of a discrete variable}\label{section5}

In this section, we present a particular interesting example of $S_{n}(a,b,c,d;x)$ that generates all classical symmetric orthogonal polynomials of a discrete variable studied and analyzed in \cite{MR2033351} and is different from the two introduced hypergeometric sequences.

If $a=-2(b+2c+4d)$ in the main difference equation  (\ref{eq:16}), after simplification of a common factor we get
\begin{multline*}
\left( x^2 (b+2 c+4 d)+x (c+2 d)+d \right) \Delta \nabla \phi_{n}(x) -2 x (c+2 d) \Delta \phi_{n}(x) \\+ n (-(b (n-1)+2 (n-2) (c+2 d))) \phi_{n}(x)=0,
\end{multline*}
which is the same equation as analyzed in \cite{MR2033351},
\begin{equation*}
(\hat{a}x^{2}+\hat{b}x+\hat{c}) \Delta \nabla y_{n}(x) - 2 \hat{b} x \Delta y(x) + n(\hat{a}(1-n)+2\hat{b}) y(x)=0,
\end{equation*}
for
\begin{equation*}
d=\hat{c}, \quad c=\hat{b}-2\hat{c}, \quad b=\hat{a}-2\hat{b}, \quad a=-2\hat{a}.
\end{equation*}

Hence, the particular polynomial
\begin{multline}\label{eq:64}
\sss{n}{-2\hat{a}}{\hat{a}-2\hat{b}}{\hat{b}-2\hat{c}}{\hat{c}}{x}
=x^{\sigma_{n}} \sum_{j=0}^{[n/2]} (-1)^{j} \binom{[n/2]}{j} \\
\times\left( \prod_{i=j}^{[n/2]-1} \frac{(2 i+2 \sigma_{n}+1) \left(-\hat{a} (i+\sigma_{n})^2+\hat{b}
   (i+\sigma_{n})-\hat{c}\right)}
{2\left(\hat{b}+\hat{a}(\sigma_{n+1}-i-[n/2]) \right)-\hat{a}} \right)
(\sigma_{n}-x)_{j} (\sigma_{n}+x)_{j},
\end{multline}
generates all cases of classical symmetric orthogonal polynomials of a discrete variable respectively as follows:
\begin{enumerate}
\item[Case 1.] If $\hat{a}=0$, $\hat{b}=1$ and $\hat{c}$ free in (\ref{eq:64}) then $\sigma(x)=(\hat{a}x^{2}+\hat{b}x+\hat{c})=x+\hat{c}$ has one real root and the symmetric Kravchuk polynomials \cite{MR1281365,MR1766991} are therefore derived as
\begin{multline*}
\sss{n}{0}{-2}{1-2\hat{c}}{\hat{c}}{x}=2^{-n}(-2\hat{c})_{n}\hyper{2}{1}{-n,-\hat{c}-x}{-2\hat{c}}{2} \\
=k_{n}^{(1/2)}(x+\hat{c};2\hat{c}),
\end{multline*}
which are orthogonal with respect to the weight function
\begin{equation*}
\varrho(x)=\frac{1}{\Gamma(\hat{c}-x+1)\Gamma(\hat{c}+x+1)} \quad \text{for } x\in\{-\hat{c},-\hat{c}+1,\dots,\hat{c}-1,\hat{c}\},
\end{equation*}
when $2\hat{c} \in {\mathbf{N}}$.

\item[Case 2.] If $\hat{a}=1$, $\hat{b}$ free and $\hat{c}=\hat{b}^{2}/4$ in (\ref{eq:64}), then $\sigma(x)=(\hat{a}x^{2}+\hat{b}x+\hat{c})=(x+\hat{b}/2)^2$ has a double real root and the symmetric Hahn-Eberlein polynomials \cite{MR1149380}
\begin{multline*}
\sss{n}{-2}{1-2\hat{b}}{\hat{b}(2-\hat{b})/2}{\hat{b}^{2}/4}{x}=\frac{((-\hat{b})_{n})^{2}}{(n-2\hat{b}-1)_{n}} \\
\times \hyper{3}{2}{-n,n-2\hat{b}-1,-x-\hat{b}/2}{-\hat{b},-\hat{b}}{1}=\tilde{h}_{n}^{(0,0)}(x+\hat{b}/2,\hat{b}+1),
\end{multline*}
are orthogonal with respect to the weight function
\begin{equation*}
\varrho(x)=\frac{1}{\Gamma^{2}(x+1+\hat{b}/2) \Gamma^{2}(-x+1+\hat{b}/2)},
\end{equation*}
for $x\in\{-\hat{b}/2,-\hat{b}/2+1,\dots,\hat{b}/2-1,\hat{b}/2\}$ when $b \in {\mathbf{N}}$.

\item[Case 3.] If $\hat{a}=1$, $\hat{b}=-\delta_{1}-\delta_{2}$ and $\hat{c}=\delta_{1}\delta_{2}$ in (\ref{eq:64}), then $\sigma(x)$ has two different real roots $\delta_{1}$ and $\delta_{2}$ and depending on the values $\delta_{1}$ and $\delta_{2}$, the symmetric Hahn-Eberlein polynomials  \cite{MR2033351,MR1149380} are derived as
\begin{multline*}
\sss{n}{-2}{1+2(\delta_{1}+\delta_{2})}{-\delta_{2}-\delta_{1}(1+2\delta_{2})}{\delta_{1}\delta_{2}}{x}\\
=\frac{(\delta_{1}+\delta_{2})_{n}\,(2\delta_{2})_{n}}{(n+2(\delta_{1}+\delta_{2})-1)_{n}} \hyper{3}{2}{-n,n+2(\delta_{1}+\delta_{2})-1,\delta_{2}-x}{\delta_{1}+\delta_{2},2 \delta_{2}}{1}\\=\tilde{h}_{n}^{(\delta_{2}-\delta_{1},\delta_{2}-\delta_{1})}(x-\delta_{2};1-2\delta_{2}),
\end{multline*}
when $-2\delta_{2}\in {\mathbf{N}}$ or the symmetric Hahn polynomials \cite{MR2033351,MR1149380} as
\begin{multline*}
\sss{n}{-2}{1+2(\delta_{1}+\delta_{2})}{-\delta_{2}-\delta_{1}(1+2\delta_{2})}{\delta_{1}\delta_{2}}{x}\\ =\frac{(\delta_{1}+\delta_{2})_{n} \, (2\delta_{1})_{n}}{(n+2(\delta_{1}+\delta_{2})-1)_{n}}  \hyper{3}{2}{-n,n+2(\delta_{1}+\delta_{2})-1,\delta_{1}-x}{\delta_{1}+\delta_{2},2\delta_{1}}{1} \\
=h_{n}^{(\delta_{1}+\delta_{2}-1,\delta_{1}+\delta_{1}-1)}(x-\delta_{1};1-2\delta_{1}),
\end{multline*}
when $-2 \delta_{1} \in {\mathbf{N}}$. Finally, for $\delta_{1}+\delta_{2}=1$ we get to Gram polynomials \cite{MR895822}.

\end{enumerate}

\end{document}